\documentclass[12pt, letter]{amsart}
\usepackage{amsmath,amssymb,amsthm,graphicx, url, verbatim, float, enumerate, picture, thm-restate, thmtools, mathtools}

\usepackage[margin=1in]{geometry}
\usepackage{tikz-cd}
\usetikzlibrary{cd} 
\usepackage{color}
\declaretheorem[name=Theorem,numberwithin=section]{thm}

\theoremstyle{definition}
\newtheorem{defn}{Definition}[section]
\newtheorem{rem}[defn]{Remark}

\theoremstyle{plain}

\newtheorem{conj}[defn]{Conjecture}

\newtheorem{lem}[defn]{Lemma}

\newcommand{\dsys}{\mathcal{A}}

\newcommand{\Kh}{\mathcal{KH}} 
\newcommand{\CKh}{\mathcal{CKH}} 
\newcommand{\JW}{\mathbf{P}}

\newcommand{\sixj}{ \left\{ \begin{array}{ccc}
2n & 2n & 2n \\ 
2n & 2n & 2n \\ 
\end{array}\right\} 
}

\newcommand{\tetran}{ \left\langle \begin{array}{ccc}
2n & 2n & 2n \\ 
2n & 2n & 2n \\ 
\end{array}\right\rangle 
}

\newcommand{\jw}{\vcenter{\hbox{\includegraphics[scale=.1]{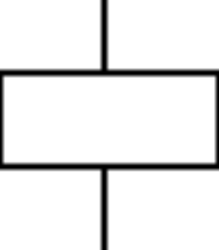}}}}

\newcommand{\jwc}{\vcenter{\hbox{\includegraphics[scale=.15]{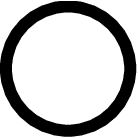}}}}
\newcommand{\kbsrc}
{\vcenter{\hbox{\includegraphics[scale=.2]{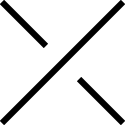}}}}
\newcommand{\kbsrh}
{\vcenter{\hbox{\includegraphics[scale=.2]{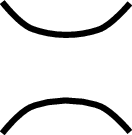}}}}
\newcommand{\kbsrv}
{\vcenter{\hbox{\includegraphics[scale=.2]{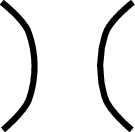}}}}
\newcommand{\kcircle}
{\vcenter{\hbox{\includegraphics[scale=.2]{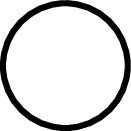}}}}

\makeatletter
\newcommand{\vast}{\bBigg@{3.5}}
\newcommand{\Vast}{\bBigg@{5}}

\begin{document}
\title{Stable Khovanov homology and Volume}
\author[C. Lee]{Christine Ruey Shan Lee}

\address[]{Department of Mathematics, Texas State University}
\email[]{vne11@txstate.edu}

\begin{abstract} We show the $n$ colored Jones polynomials of a highly twisted link approach the Kauffman bracket of an $n$ colored skein element. This is in the sense that the corresponding categorifications of the colored Jones polynomials approach the categorification of the Kauffman bracket of the skein element in a direct limit, as the number of full twists of each twist region tends toward infinity, proving a quantum version of Thurston's hyperbolic Dehn surgery theorem implicit in Rozansky's work, and giving a categorical version of a result by Champanerkar-Kofman. In view of the volume conjecture, we compute the asymptotic growth rate of the Kauffman bracket of the limiting skein element at a root of unity and relate it to the volumes of regular ideal octahedra that arise naturally from the evaluation of the colored Jones polynomials of the link. 
\end{abstract}

\maketitle


\section{Introduction} 

To a link $L$ in $S^3$, the colored Jones function assigns a sequence of Laurent polynomials $\{\widehat{J}_{n+1}(L; A)\}_{n=1}^{\infty}$ called the colored Jones polynomials\footnote{We use the Kauffman variable $A$ here rather than $t$ with $A = t^{-1/4}$, and we use \ $\widehat{}$ \ to indicate the reduced versions of the polynomials. We index the polynomials by the dimension of the corresponding irreducible representations of $U_q(\mathfrak{sl}_{2})$. See Definition \ref{d.cjp}.}, where for $n=1$, $\widehat{J_2}(L; A)$ is the Jones polynomial. The celebrated volume conjecture predicts a close relationship between the exponential growth rate of the colored Jones polynomials at a root of unity and the volume of a hyperbolic knot. 

\begin{conj}[\cite{Kashaev}, \cite{MM}] \label{c.volume}
 Let $K$ be a hyperbolic knot and $N = n+1$, then  
\[ 2\pi \lim_{N \rightarrow \infty} \frac{\log |\widehat{J}_N(K; e^{\frac{\pi i}{2N}})|}{N} = vol(S^3 \setminus K). \] 
\end{conj}

See \cite{MM} for the general version of the conjecture defined for all knots using simplicial volume, \cite[Section 3.3]{M} for an introduction, and \cite{chen} for an account of  current progress on the conjecture.   Despite the progress made, the general problem remains intractable. Nevertheless, if the conjecture is true, it would have many consequences for the relationship between quantum invariants related to the colored Jones polynomials and hyperbolic geometry, and so, studying whether related quantum invariants behave in a way that is suggested by the conjecture can be a fruitful way to understand the volume conjecture. 

We are interested in a ``quantum" analogue of the following version of Thurston's hyperbolic Dehn surgery theorem \cite{ThurstonW}. Let $M$ be a manifold with $t$ cusps, and let $M_k = M_{(k_1, \ldots, k_t)}$ for $k = (k_1, \ldots, k_t)$, with $k_i \in \mathbb{Z} \cup \infty$, denote the manifold resulting from performing a $-1/k_i$-Dehn surgery on the $i$th cusp of $M$ for $1\leq i \leq t$. 

\begin{thm}[Hyperbolic Dehn surgery] \label{t.hsurgery}
Let $k = (k_1, \ldots, k_t)$ and let $k\rightarrow \infty$ denote the limit as $k_1 \rightarrow \infty$, $k_2 \rightarrow \infty$, \ldots, $k_t \rightarrow \infty$. Then
\[ \lim_{k\rightarrow \infty} vol(M_{k})  = vol(M) \qquad \text{ and } \qquad vol(M_{k}) < vol(M) \qquad (k\not=\infty). \]  
\end{thm} 

Let $L_\infty = L_{\infty}(D)$ be the fully augmented link obtained from a diagram $D$ of $L$ by enclosing every twist region by a crossing circle \cite{Purcell}. For $k = (k_1, \ldots, k_t)$ let $L_k = L_k(D)$ be the link represented by the diagram with $k_i$ additional positive full twists in each twist region obtained by performing a $-1/k_i$ surgery on each crossing circle of $L_{\infty}(D)$. In the setting of Theorem \ref{t.hsurgery} where $M = S^3 \setminus L_{\infty}$ and $M_{k} = S^3 \setminus L_{k}$,  
it is natural to consider whether the colored Jones polynomials, or related invariants,  behave in an analogous way as predicted by the volume conjecture. For the colored Jones polynomials, Champanerkar and Kofman \cite{CK05} showed their Mahler measure converges, and their coefficient vectors decompose into fixed blocks that separate as $k\rightarrow \infty$. 

Since the volume conjecture is known for many fully augmented links by van der Veen \cite{vdV2009}, it would be useful to have a notion of convergence for the colored Jones polynomials for which one can compare the behavior of the limiting object at the root of unity considered by the conjecture to $vol(S^3 \setminus L_{\infty})$. With this goal in mind, the purpose of this article is to state and prove a quantum analogue of Theorem \ref{t.hsurgery} for the categorification of the colored Jones polynomials, by defining convergence in terms of the direct limit of the corresponding link homology groups. 

With $L_k = L_{(k_1, k_2, \ldots, k_t)}(D)$ as above let $L^n_J$ be the skein element obtained by taking the $n$-blackboard cable of $L_k$ and replacing every $n$-cabled twist region $\vcenter{\hbox{\includegraphics[scale=.15]{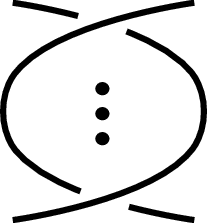}}}$  by a $2n$ Jones-Wenzl projector $\vcenter{\hbox{ \includegraphics[scale=.15]{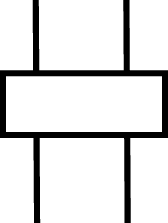}}}$. See Definition \ref{d.limitingskein}. Denote by $\{ \Kh_{i, j}(L_{k}, n) \}$ the set of $n$ colored Khovanov homology groups, which categorifies the unreduced $n+1$ colored Jones polynomial $J_{n+1}(L_k; A)$. Similarly, let $\{ \Kh_{i, j} (L_{J}, n) \}$ denote the categorification of the Kauffman bracket $\langle L^n_J\rangle$ of the skein element $L^n_J$. We state the first result. 

\begin{restatable}{thm}{tstable} \label{t.stability}
Fix an integer $n\geq 1$.  The set of homology groups $\{\Kh_{i, j}(L_{k}, n)\}$ forms a direct system as $k \rightarrow \infty$, and its direct limit is given by $\{\Kh_{i, j}(L_{J}, n)\}$.
\end{restatable} 

For the precise notion of direct limit in this setting, see Section \ref{ss.stableKhtorus}, where we prove Theorem \ref{t.stability}. If a set of homology groups has a direct limit in our notion, then its graded Euler characteristic also approaches that of the graded Euler characteristic of the homology in the direct limit. Therefore, a consequence of Theorem \ref{t.stability} is that the unreduced $n+1$ colored Jones polynomials of the link $L_k$ approaches that of the Kauffman bracket of $L^n_J$:  $J_{n+1}^{\infty}(L; A):=  \langle L^n_J \rangle$ under twisting, and similarly for the reduced versions. In view of the volume conjecture and Thurston's hyperbolic Dehn surgery theorem, it is natural to ask whether $J_{n+1}^{\infty}(L; A)$ is related to the hyperbolic volume of the fully augmented link $L_{\infty}$. Taking the limit as $n+1\rightarrow \infty$ of $\widehat{J}_{n+1}^{\infty}(L; A)$ at the root $e^{\frac{\pi i}{2N-1}}$ near the root of unity $e^{\frac{\pi i}{2N}}$ considered for the conjecture, we show the following in Section \ref{s.volumeinfo}. 

\begin{restatable}{thm}{tvolume}  \label{t.volume} Let $N = n+1$, and let $L$ be a link with diagram $D$, with $L_{\infty} = L_\infty(D)$ the associated fully augmented link.  Suppose there is a sequence of moves containing only the triangle moves to obtain $L^n_J$ from the $2n$ colored theta graph $\vcenter{\hbox{\def \svgwidth{.04\columnwidth} \includegraphics[scale=.1]{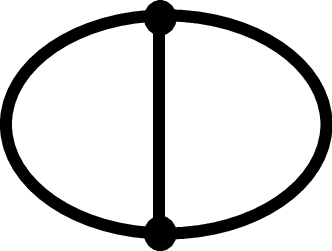}}}$.  Let $T$ be the number of triangular KTG moves in such a sequence. Then
\[2\pi \lim_{N\rightarrow \infty} \frac{\log |\widehat{J}_N^{\infty}(L; e^{\frac{\pi i}{2N-1}})|}{2N-1}   = 2Tv_8, \]
where $v_8$ is the volume of a regular ideal octahedron. 
\end{restatable} 

If one obtains the theta graph without applying any triangle moves, as in the case of the standard diagram of the 3-tangle pretzel knots, then the limit on the left hand side is 0. Here KTG stands for Knotted Trivalent Graphs, which appear in evaluations of the colored Jones polynomials using graphical skein-theoretic calculus \cite{MV94}, \cite{Lickorish}, see Definition \ref{d.ktg}.  Theorem \ref{t.volume} can be compared with van der Veen's result \cite{vdV2009} on a version of the volume conjecture for augmented knotted trivalent graphs, which includes many fully augmented links as a special case. Theorem \ref{t.volume} may also be of interest to a generalized version of the volume conjecture by Gukov \cite{Gukov} that predicts the limit of the colored Jones polynomial near the original root of unity captures the volume of noncomplete hyperbolic structures on $S^3\setminus L$ for $L$ a knot.  Note not every link, for example the 4-tangle pretzel link with the standard diagram, satisfies the conditions of Theorem \ref{t.volume}, which require that there exists a sequence of moves containing only triangle moves to obtain $L^n_J$ from the theta graph. An example of a link $L$ satisfying the conditions of Theorem \ref{t.volume} is shown in Figure \ref{f.link}. 

\begin{figure}[H] 

\includegraphics[scale=.3]{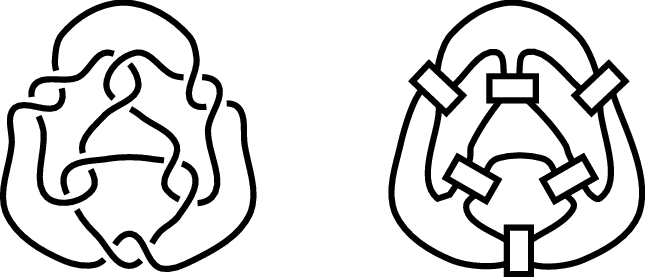}
\caption{\label{f.link} A link $L$ and $L^1_J$ satisfying the conditions of Theorem \ref{t.volume}.}
\end{figure} 

Recall $t$ is the number of twist regions  of a diagram $D$ of $L$, and the number of crossing circles in the associated fully augmented link $L_{\infty} = L_{\infty}(D)$. Since the volume of the fully augmented link $L_{\infty}$ is at least $2(t-1)v_8$ \cite[Proposition 3.1]{FKP}, ie,  $vol(S^3 \setminus L_{\infty}) \geq 2(t-1)v_8$, the asymptotics of the limiting skein element $L^n_J$ we consider appears to only capture partial volume information. This indicates a potential discrepancy between categorical convergence and analytical convergence, which we plan to study in a future project.

\subsection*{Acknowledgments}
I would like to thank Neil Hoffman for the conversation that inspired this paper.  I am also grateful to Roland van der Veen for illuminating conversations about his work, and I would like to acknolwedge the partial support by NSF grant No. 2244923.

\section{Background}

\subsection{The colored Jones polynomials and the Temperley-Lieb algebra}
Fix $n\geq 1$. The $n$ Temperley-Lieb algebra, denoted by $TL_n$, is the $\mathbb{C}$-vector space generated by properly embedded tangle and link diagrams in a disk $D^2$ viewed as a square, with $n$ marked points on the top boundary and the bottom boundary. The diagrams are considered up to isotopy rel the boundary and the Kauffman bracket skein relations:  
\begin{equation} \label{e.kbsr} \kbsrc = A \ \kbsrv  + A^{-1}  \ \kbsrh\text{, and }\kcircle  = -A^2 -  A^{-2}. \end{equation}  
We refer to an element in $TL_n$ as a \textit{skein element}. The natural multiplication operation $T_1 \cdot T_2$ between two skein elements $T_1$ and $T_2$ in $TL_n$ identifies the $n$ points on the bottom boundary of $T_1$ with $n$ points on the top boundary of $T_2$, and turns $TL_n$ into an algebra generated by $\{e_i(n)\}_{i=1}^{n-1}$: 
\begin{center} 
$ e_i(n) =  \vcenter{\hbox{\includegraphics[scale=.5]{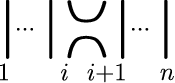}}} $
\end{center}   

The $n$ Jones-Wenzl projector is an element in $TL_n$ defined as follows.  

\begin{defn}\cite[Lemma 13.2]{Lickorish}\label{d.jw}
Let $A$ be an indeterminate and $|_n$ denote $n$ parallel strands. There is a unique skein element $\jw_n$ in $TL_n$, called the \textit{$n$ Jones-Wenzl projector}, such that
\begin{enumerate}[(i)]
\item $\jw_n \cdot e_i(n)  = 0 =  e_i(n) \cdot \jw_n $, 
\item $\jw_n - |_n$ belongs to the algebra generated by $\{e_i(n) \}_{i=1}^{n-1}$, 
\item $\jw_n \cdot \jw_n = \jw_n$, and 
\item  $\langle \kcircle_n \rangle = (-1)^n \frac{A^{2(n+1)} - A^{-2(n+1)}}{A^2 - A^{-2}}$. 
\end{enumerate}
\end{defn}

\noindent From the defining properties we obtain the ``projector-absorbing" property shown in Figure \ref{f.projabsorb}. 
\begin{figure}[H]
\includegraphics[scale=.7]{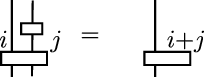} 
\caption{\label{f.projabsorb} A larger projector absorbs a smaller projector.}
\end{figure} 

We similarly consider the Kauffman bracket skein module of isotopy classes of link diagrams in the 2-sphere $S^2$  quotiented by the Kauffman bracket skein relations \eqref{e.kbsr}. 
Let $L$ be a skein element with empty boundary in $S^2$, i.e., a link diagram. The Kauffman bracket $\langle L \rangle$ is the rational function multiplying the empty diagram after reducing $L$ to the empty diagram using the Kauffman bracket skein relations. 

For two tangles $T_1 \subset D^2_1$ and $T_2 \subset D^2_2$ in $TL_n$ arranged as in Figure \ref{f.join}, we  define the join $T_1 \star T_2$ in the plane which joins the top $n$ points on $D^2_1$ to the top $n$ points on $D^2_2$, and joins the bottom $n$ points on $D^2_1$ to the bottom $n$ points on $D^2_2$ by $n$ parallel arcs in the plane. 

\begin{figure}[H]
\includegraphics[scale=.7]{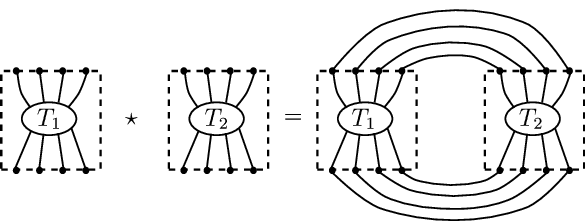} 
\caption{\label{f.join} Joining two tangles $T_1$ and $T_2$ in $TL_n$. Here $n=4$.}
\end{figure}

Let $D$ be the diagram of a link $L$ with $\ell$ components. For $n\geq 1$ we define the reduced $n+1$ colored Jones polynomial $\widehat{J}_{n+1}(L; A)$. 
For each component remove the intersection with a disk intersecting the diagram and the component in a simple arc, then take the $n$-blackboard cable of the remaining tangle, call this $\cup_{i=1}^\ell \mathbb{T}_i^n$. Let $\mathbb{D}^n =  \cup_{i=1}^\ell (\mathbb{T}_i^n \star \jw_n) $ be the skein element resulting from joining the $n$ parallel strands from each $n$-cabled link component with an $n$ Jones-Wenzl projector. 

\begin{defn} \label{d.cjp} Let $L$ be an oriented link with diagram $D$ and let $w(D)$ denote its writhe. 
The (reduced) \textit{$n+1$ colored Jones polynomial} $\widehat{J}_{n+1}(L; A)$ of a link $L$ is defined by
\begin{equation} \label{e.cjp}  \widehat{J}_{n+1}(L; A) := \left((-1)^nA^{(n^2+2n)} \right)^{-w(D)} \langle \mathbb{D}^n \rangle / \langle \jwc_n \rangle. \end{equation}
The \textit{unreduced $n$ colored Jones polynomial} is obtained from the reduced version by multiplying back the Kauffman bracket of the $n$ colored unknot: 
\[ J_n(L; A) = \widehat{J}_n(L; A) \cdot \langle \jwc_n \rangle. \]   
\end{defn}

\paragraph{\textbf{Note on convention.}}
With our normalization, we have $\widehat{J}_{n+1}(U; A) = 1$ for the unknot $U$. Note $n$ denotes the color on the strand and $N = n+1$ denotes the dimension of the corresponding irreducible representation, as is done in \cite{vdV2009}. 

\subsection{Khovanov homology}
We use Bar-Natan's formulation of Khovanov homology in \cite{BN} and loosely follow \cite{CK} in notation. Let $Cob(n)$ be the additive category defined by: \\
$\bullet$ Objects: Isotopy classes of formally $q$-graded $TL_n$ skein elements. Here $-q =  A^{- 2}$. \\
$\bullet$ Morphisms: Elements in the free $\mathbb{Z}$-module spanned by isotopy classes of orientable surface cobordisms between two $q$-graded $TL_n$ skein elements, quotiented by the Bar-Natan skein relations depicted in Figure \ref{f.cobrelations} , where a dot on a cobordism represents a handle.  

\begin{figure}[H]
\includegraphics[scale=.7]{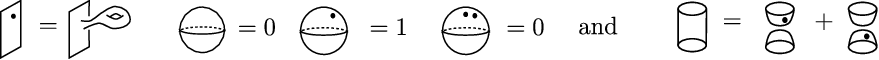}
\caption{\label{f.cobrelations} Relations on cobordisms in $Cob(n)$.}
\end{figure} 

\begin{defn} Define $Kom^+(n)$ to be the category of semi-infinite positive chain complexes of objects and morphisms in $Cob(n)$. That is, we allow chain complexes which may be unbounded in positive homological degrees. 
\end{defn} 

Allowing chain complexes which may be unbounded in positive homological degrees  is not used in the definition of Khovanov homology, but will be relevant in the definition of colored Khovanov homology in the next section.

Given an oriented link or tangle diagram $D$, form the underlying Khovanov  complex $\{ \CKh_{i, *}(D)\}$ using the categorified versions of the skein relations in Figure \ref{f.skeincat}: 
\begin{figure}[H]
\includegraphics[scale=.7]{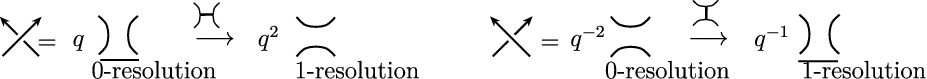}
\caption{\label{f.skeincat} Categorified skein relations.}
\end{figure} 
The underlined diagrams in Figure \ref{f.skeincat} have homological degree $i = 0$ and the homological grading extends to the rest of the diagram where a resolution is chosen at every crossing. The quantum grading $j$ is given by the degree of the monomial multiplying the diagram in $q$. Here  $\vcenter{\hbox{\includegraphics[scale=.7]{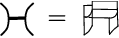}}}$, the saddle cobordism from $\kbsrv$ to $\kbsrh$, reading from bottom to top. 

The skein  relation in terms of $q$ comes from the Kauffman bracket skein relation in $A$ by factoring $A$ out of \eqref{e.kbsr}. The variable $A^{-2}$ is then replaced by $-q$ for the new skein relation. 
\[ \kbsrc = A\left(\ \kbsrv  + A^{-2}  \ \kbsrh \right) = A\left(\ \kbsrv  - q  \ \kbsrh \right) \text{, and }\kcircle  = (q+q^{-1}). \] 

The variable $A$ that is factored out is combined with the writhe term multiplying the Kauffman bracket in \eqref{e.cjp} and written in terms of $q$. This leads to the $q$ coefficients in the categorified skein relations of Figure \ref{f.skeincat}.

Formally, let $c$ be the number of crossings in $D$, and pick an ordering of the crossings by numbering them $1, \ldots, c$. First we form $\{\CKh_{i, j}(D)\}$. To a bit string $s$ in $\{0, 1\}^c$ with $s(k)$ the $k$th digit of the string $s$, associate a diagram $D_{s}$ consisting of a disjoint collection of circles and arcs in the plane, resulting from applying the Kauffman state $s$ that chooses the $s(k)$-resolution at the $k$th crossing for $1\leq k\leq c$.  Multiply $D_{s}$ by $q$ raised to the appropriate power for each crossing following Figure \ref{f.skeincat} and still denote the resulted shifted complex by $D_{s}$. Denote the homological grading of $D_s$ by $h(D_s)$, then define $\CKh_{i, *}(D) := \oplus_{h(D_s) = i} D_s$. 
 
For the differential, let $\widetilde{d}_k$ be the saddle cobordism map that maps the $0$-resolution to the $1$-resolution at the $k$th crossing, and let $g_s(k)$ be the number of 1's in the string $s$ before $s(k)$. Define 
\[ \widetilde{d}(D_s) = \sum_{k: s(k) = 0} (-1)^{g_s(k)} D_{\widetilde{d}_k(s)}.\]

Extend $\widetilde{d}$ linearly over $\CKh_{i, *}(D) = \oplus_{h(D_s) = i} D_s$ for each homological grading $i$. 
If $D$ is a tangle diagram, the homotopy type of the chain complex $\{\CKh_{i, j}(D), \widetilde{d}\}$ is an invariant of the isotopy class of the tangle \cite{BN}. If $D$ is a link diagram, then $D_s$ for every $s$ is a disjoint collection of only circles without arcs, and we compose $\widetilde{d}$ with the delooping map $p$ to send $D_s$ to a $\mathbb{Z}$-module $V_s$. 

\begin{figure}[H]
\includegraphics[scale=.7]{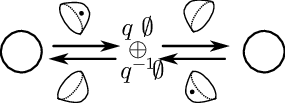}
\caption{\label{f.deloop} Delooping maps}
\end{figure} 

Let $|D_s|$ denote the number of circles in $D_s$, then $p(D_s) = V_s = (q\mathbb{Z} \oplus q^{-1} \mathbb{Z}) \otimes \cdots \otimes (q\mathbb{Z} \oplus q^{-1} \mathbb{Z}) = (q\mathbb{Z} \oplus q^{-1} \mathbb{Z})^{\otimes |D_s|} $.
Let $v\in V_s$ and extend $p, p^{-1}$ linearly over $V_s$, then $\widetilde{d}$ is extended to a map $d$: 
 \[ d(v) = \sum_{k: s(k) = 0} (-1)^{g_s(k)} p(p^{-1}(v)_{\widetilde{d}_k(s)}).\]

If $D$ is the diagram of a link $L$, define $\{ \CKh_{i, *}(D)  := \oplus_{h(D_s) = i} V_s \}$ and let the quantum grading $j = q(v)$ of an element $v$ be the power of the $q$-coefficient multiplying $v$. The map $d$ preserves the quantum grading while increasing the homological grading by $1$. Extend $d$ linearly over $\CKh_{i, *}(D) = \oplus_{h(D_s) = i} V_s$. The set of bigraded homology groups $\{ \Kh_{i, j}(D) \}$ of the chain complex $\{ \CKh_{i, j}(D), d \}$ is an invariant of $L$. The decategorifcation of $\Kh(D)$ by taking the graded Euler characteristic recovers the (unreduced) Jones polynomial $J(L; A)$ of $L$: 
\[ J_2(L; A) = J(L; A) = \sum_{i, j} (-1)^i \left( q^j rank(\Kh_{i, j}(D)) \rvert_{-q = A^{-2}} \right) .   \] 

Up to a degree shift by the writhe of the oriented diagram $D$, the chain complex $\{\CKh_{i, j}(D), d\}$ also gives a categorification of the Kauffman bracket of $D$. 

\subsection{Colored Khovanov homology}
Let $D$ be a diagram of a link $L$ and consider the skein element $\mathbb{D}^n = \cup_{i=1}^\ell (\mathbb{T}_i^n \star \jw_n) $ as in Definition \ref{d.cjp} for the colored Jones polynomials. 
For each $n\geq 1$, the chain complex categorifying the $n+1$ colored Jones polynomial of the link $L$, denoted by $\CKh(L, n) = \{\CKh_{i, j}(D, n) , d \} $, is constructed by tensoring the Khovanov chain complex $\{ \CKh(\cup_{i=1}^\ell \mathbb{T}_i^n), \widetilde{d} \}$ of the tangle $\cup_{i=1}^\ell \mathbb{T}_i^n$, with the categorification of the Jones-Wenzl projector, see \cite{Roz12}, \cite{CK}, denoted by $\JW_n$. For each $n\geq 1$, the resulting set of bigraded $n$ colored Khovanov homology groups $\{ \Kh_{i, j}(L, n) \}$ recovers the $n+1$ colored Jones polynomial after decategorification. 
\[  J_{n+1}(L; A) = \sum_{i, j} (-1)^i \left( q^j rank(\Kh_{i, j}(L, n)) \rvert_{-q = A^{-2}} \right) .   \]

In this paper, we use Rozansky's categorification of the $n$ Jones-Wenzl projector \cite{Roz12}. Since the categorification by Rozansky satisfies categorified versions of the defining properties Definition \ref{d.jw} (i) - (iv) of the Jones-Wenzl projector, it is homotopy equivalent to the categorification defined by Cooper-Krushkal \cite{CK}. See  \cite[Corollary 3.5]{CK} for the theorem that implies this statement.

We use the categorified version of the projector-absorbing property (Figure \ref{f.projabsorb}). 

\begin{lem}{(\cite[Proposition 3.6]{Roz14})} \label{l.projabsorb}
Let $\JW_j$ denote the categorification of the $j$ Jones-Wenzl projector, and let $\JW_j \cdot \JW_{i+j}$ denote the complex obtained by tensoring which categorifies $\jw_j \cdot \jw_{i+j}$. Then $\JW_j \cdot \JW_{i+j}$ is homotopy-equivalent to $\JW_{i+j}$. That is, 
$\JW_j \cdot \JW_{i+j} \simeq \JW_{i+j}$. 

\end{lem}

\section{A quantum Dehn surgery theorem}
\subsection{The stable Khovanov homology of infinite torus braids} \label{ss.stableKhtorus}

We follow Rozansky \cite{Roz12} and Islambouli-Willis \cite{IW} in the description of the stable Khovanov homology of infinite torus braids, see \cite{GOR} for further references. 
\begin{defn} 
We work in the homotopy category of complexes in \\$Kom^+(n)$. For $m\geq 0$ let $O^h_+(m)$ denote a complex which starts at the $m$th homological degree: 
\[ O^h_+(m) = (A_{m} \to A_{m+1} \to \cdots).  \] 
The \textit{homological order} of a complex $A$, denoted by $|A|_h$, is 
\[  |A|_h = \inf \{m: A\simeq O^h_+(m) \}. \] 
\end{defn} 

\begin{rem}
We are working with dual versions of the complexes of \cite{Roz12, Roz14}. To obtain the complex in \cite{Roz12}, change every homological degree to its negative and reverse the direction of every arrow. For example, the complex $O^h_-(m)$ from \cite{Roz12} which ends at the $-m$th homological degree is given by dualizing $O^h_+(m)$. 
\[ O^h_-(m) = (\cdots \to A_{-(m+1)} \to A_{-m}).  \] 
\end{rem} 

A \textit{direct system} is a sequence of complexes $\{A^i\}$ connected by chain morphisms 
$\mathcal{A} =  A^0 \stackrel{f^0}{\to} A^1 \stackrel{f^1}{\to}\cdots$. Rozansky defines a direct system to be \textit{Cauchy} if $ \displaystyle \lim_{i\rightarrow \infty}  |Cone(f^i)|_h = \infty$.

\begin{defn} \label{d.dlimit} 
 A direct system has a limit: $\varinjlim \dsys = A$, where $A$ is a chain complex, if there exist chain morphisms $A^i \stackrel{\widetilde{f}^i}{\to} A$ such that they form commutative triangles
\begin{center} 
\begin{tikzcd}								
  A^i \arrow[urrd, bend left, "\widetilde{f^i}"]  \arrow[r, "f^i"']   & A^{i+1} \arrow[r, "\widetilde{f}^{i+1}"']  & A  
\end{tikzcd}
\end{center} 

where $\widetilde{f^i} \sim \widetilde{f}^{i+1}\circ f^i$ and $\displaystyle \lim_{i\to \infty} | Cone(\widetilde{f}^i)|_h = \infty$.   \end{defn} 

Rozansky showed that a direct system has a limit if and only if it is Cauchy \cite[Theorem 2.5]{Roz12} and if a direct system has a limit, then it is unique \cite[Theorem 2.6]{Roz12}. In \cite{Roz12}, he defines a direct system of shifted Khovanov complexes of torus braids on $n$ strands connected by chain morphisms, which he shows to be Cauchy. Therefore, the direct system has a unique limit  $\JW_n$, called \textit{the stable homology group of infinite torus braids on $n$ strands}. His results generalize those of Sto\v{s}i\'{c}'s on the stable Khovanov homology of infinite torus links \cite{stosic}. He also shows $\JW_n$ categorifies the $n$ Jones-Wenzl projector $\jw_n$. 

In \cite{IW},  Islambouli and Willis showed that the Khovanov complexes of infinite torus braids considered above can be replaced by the Khovanov complexes of any complete semi-infinite positive braids in the construction to give a categorification of the $n$ Jones-Wenzl projector.

\begin{defn}
A \textit{semi-infinite positive} braid $\beta$ on $n$ strands is a semi-infinite word in the standard generators $\sigma_i$ of the braid group on $n$ strands:  $ \beta = \sigma_{j_1} \sigma_{j_2} \cdots$. A semi-infinite positive braid $\beta$ is \textit{complete} if $\sigma_i$ for each $1\leq i \leq n-1$ occurs infinitely often in $\beta$. 
\end{defn}

\begin{thm}{(\cite[Theorem 1.1]{IW})} \label{t.IW} Let $\beta$ be any complete semi-infinite positive braid on $n$ strands, viewed as the limit of positive braid words $\beta = \varinjlim \sigma_{j_1} \sigma_{j_2} \cdots \sigma_{j_\ell}$. 
Then the limiting Khovanov chain complex categorifies $\jw_n$. 
\end{thm}

We are now ready to prove Theorem \ref{t.stability}.
\subsection{Proof of Theorem \ref{t.stability}} 
 We restate Theorem \ref{t.stability} for the convenience of the reader. Recall $L$ is a link  with diagram $D$ with $t$ twist regions,  and $L_{\infty}(D)$ denotes the corresponding fully-augmented link, where there is a crossing circle $C_i$ for each twist region for $1\leq i \leq t$. Let $L_k = L_{k_1, \ldots, k_t} = L_k(D)$ denote the link obtained from $L_{\infty}$ by performing a $-1/k_i$ Dehn surgery on each crossing circle $C_i$ for $1\leq i \leq t$. Let $\Kh(L_k(D), n)$ denote the $n$ colored Khovanov homology of the link. Finally, we define $L^n_J$. 
 
 \begin{defn} \label{d.limitingskein}
Let $L$ be a link with diagram $D$, then the  skein element $L^n_J$ is the diagram  obtained by taking the $n$-blackboard cable of $D$, removing the $(n, n)$-tangle corresponding to each $n$-cabled twist region, and replacing it by joining with a $2n$ Jones-Wenzl projector. 
\end{defn} 

Let $\{\Kh_{i, j}(L_J, n)\}$ denote the categorification of the Kauffman bracket of $L^n_J$. 
 
\tstable*

\begin{proof}
It suffices to show the chain complexes $\left\{ \CKh_{i, j}(L_k, n) \right\}$ form a direct system as $k\rightarrow \infty$ with direct limit given by $\left\{ \CKh_{i, j}(L_J, n) \right\}$. 
That is, we consider the set of chain complexes $\{\CKh(L_k, n) \}= \{\CKh(L_{k_1, k_2, \ldots, k_t}, n)\}$ as $k_1 \rightarrow \infty, k_2 \rightarrow \infty, \ldots, k_t \rightarrow \infty$. Recall $\CKh(L_k, n)$ is obtained by tensoring each component of the Khovanov complex categorifying the $n$-cable of the diagram $D$ with the complex $\JW_n$ of an $n$ Jones-Wenzl projector.  For each $k_i$, $1\leq i \leq t$ starting from $i=1$, the tangle that is the $n$-cabled twist region of $k_i$ full twists is $\beta = ( \sigma_{n} \sigma_{n-1}^2 \cdots \sigma_{2}^{n-1}\sigma_{1}^{n}\sigma_{2}^{n-1} \cdots \sigma^2_{n-1}\sigma_{n})^{k_i}$, a complete semi-infinite positive braid as $k_i \rightarrow \infty$. Therefore, by Theorem \ref{t.IW}, it can be replaced by a  $2n$ Jones-Wenzl projector in the direct limit. This is shown in the first part of Figure \ref{f.dlimitabsorb}. 
\begin{figure}[H]
\includegraphics[scale=.7]{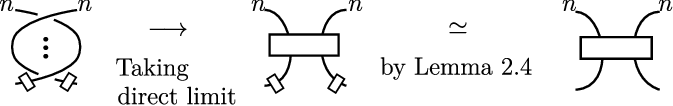}
\caption{\label{f.dlimitabsorb} From left to right: the $n$-cable of a twist region, replacing the twist region by $\JW_n$ in the direct limit, absorbing the two smaller projectors.}
\end{figure} 

Applying the projector-absorbing property, Lemma \ref{l.projabsorb}, to absorb any $n$ projector on a component of $D$, we arrive at the chain complex categorifying the skein element $L^n_J$ obtained by replacing each $n$-cabled twist region by a $2n$ Jones-Wenzl projector. 
\end{proof}

\section{Volume information from the limiting spin network} \label{s.volumeinfo}

In this section, we prove Theorem \ref{t.volume} by evaluating the reduced Jones polynomial of the limiting spin network $L^n_J$ at the root of unity $e^{\frac{\pi i}{2N-1}}$ as $N\rightarrow \infty$. We first need to set up the theory of KTGs (knotted trivalent graphs) and the graphical skein-theoretic calculus on KTGs used in the evaluation of the colored Jones polynomials.   

\subsection{Skein-theoretic moves on Knotted Trivalent Graphs}

The following graphical skein-theoretic formulas are helpful for evaluating the colored Jones polynomials as defined in Definition \ref{d.cjp}. First, we regard a skein element arranged as in the left hand side of Figure \ref{f.triv} as a trivalent vertex: 

\begin{figure}[H]
 \includegraphics[scale=.1]{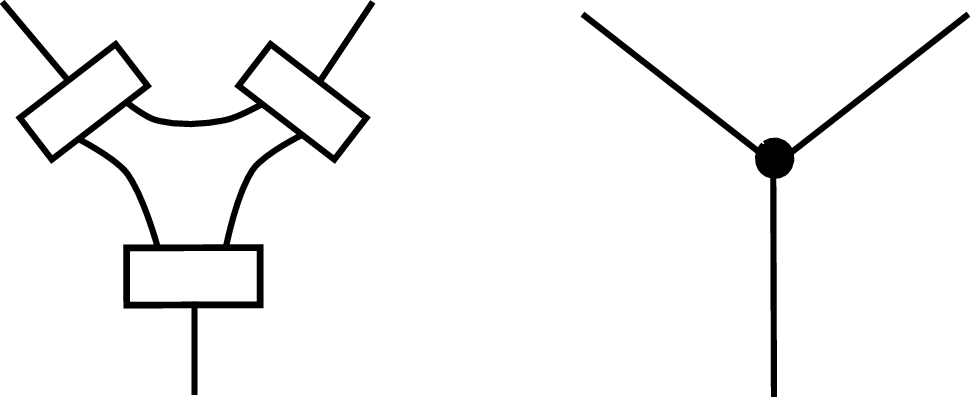}
 \caption{\label{f.triv} A trivalent vertex.}
\end{figure} 

A positive integer $m$ next to a strand indicates $m$ parallel copies (the ``color") of that strand. In $TL_n$, we have the following equivalences from  \cite{MV94}.  

\begin{figure}[H]
\includegraphics[scale=.7]{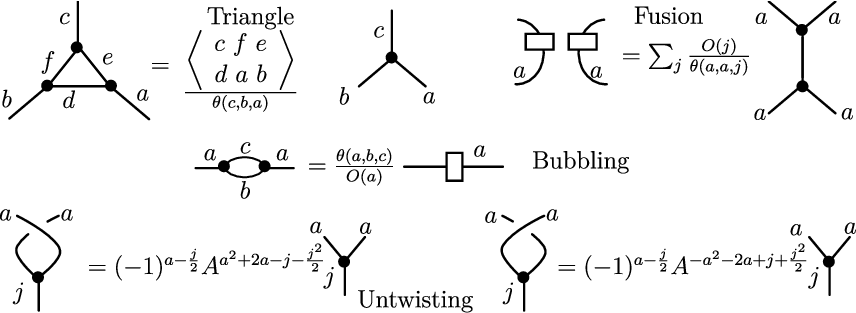}
\caption{\label{f.moves} Skein-theoretic graphical calculus.}
\end{figure}

Here $O(a) = \langle \jwc_a \rangle$.
A triple $(a, b, c)$ of colors is called \textit{admissible} if $a+b+c$ is even and $|a-b|\leq c \leq a+b$. The theta function $\theta(a, b, c) = \langle  
\vcenter{\hbox{
\includegraphics[scale=.5]{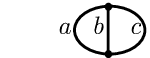}}} \ 
\rangle$ is 0 when $(a, b, c)$ is not admissible. One obtains the $6j$-symbol $\left \{ \begin{array}{ccc} a & b & c \\ d & e & f  \end{array} \right\}$ from the tetrahedron coefficient $\left \langle \begin{array}{ccc} c & b & d \\ f & e & a \end{array} \right \rangle$ by the equation 
$ \left \{ \begin{array}{ccc} a & b & c \\ d & e & f  \end{array} \right\}  = \frac{O(c) \left \langle \begin{array}{ccc} c & b & d \\ f & e & a \end{array} \right \rangle}{\theta(c, a, e) \theta(c, b, d)} $.

Trivalent graphs naturally arise from a crossingless skein element where each Jones-Wenzl projector is in an arrangement as on the left-hand side of the equation in Figure \ref{f.triv}, when we view them as trivalent vertices, and when we view  arcs between projectors as edges of a graph connecting the trivalent vertices. The Kauffman bracket of the skein element $\mathbb{D}^n$ for every link diagram $D$ can be reduced to a sum of Kauffman brackets of skein elements represented by trivalent graphs with Laurent series coefficients as we show in the next section. 
\subsection{The $n$ colored Jones polynomial as a sum of Kauffman brackets of KTGs}
We describe how the fusion and untwisting equivalences of Figure \ref{f.moves} are applied to simplify the Kauffman bracket of a link diagram. An example is illustrated in Figure \ref{f.twostrands}. \\  
(1) Fusion: Replace two strands decorated by Jones-Wenzl projectors of an $n$-cabled twist region by a sum over the  fused strands  multiplied by $\frac{O(n)}{\theta(n, n, j)}$, indexed by the fusion parameter $j$. \\
(2) Untwisting: Replace each $n$-cabled crossing by two strands using the untwisting equivalence and multiply by the corresponding rational function coefficient.  Since the Kauffman bracket extends linearly over a sum of diagrams, each term in the sum is now the Kauffman bracket of a trivalent graph $\langle \Gamma_j \rangle$, indexed by the fusion parameter $j$ for the twist region. 

\begin{figure}[H]
\[ \vcenter{\hbox{ \includegraphics[scale=.5]{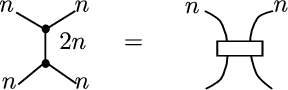}}} = \sum_{0 \leq j \text{ even } \leq 2n} \frac{O(j)}{\theta(n, n, j)} \vcenter{\hbox{ \includegraphics[scale=.7]{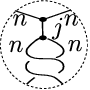}}}  
 = \sum_{0 \leq j \text{ even } \leq 2n} \frac{O(j)}{\theta(n, n, j)} (-1)^{n-\frac{j}{2}} A^{n^2 + 2n - j - \frac{j^2}{2}}  \vcenter{\hbox{ \includegraphics[scale=.7]{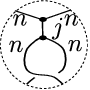}}} \]
\caption{\label{f.twostrands} Applying fusion and untwisting to a twist region.}
\end{figure} 

Recall $t$ is the number of twist regions in a diagram $D$ of a link $L$ numbered from $1, 2,\ldots, t$.  Let $c(i)$ be  the number of crossings in the twist region numbered $i$ for $1\leq i \leq t$, multiplied by a $-1$ if the crossing is positive from applying the untwisting equivalence. From steps (1) and (2) we have 
\begin{equation}\label{fe}
\langle \mathbb{D}^n \rangle =  \sum_{1\leq j_1, j_2, \ldots, j_t \text{ even } \leq 2n} \left(\prod_{i=1}^t  (-1)^{n-\frac{j_i}{2}} A^{c(i)(n^2+2n-j_i-\frac{j_i^2}{2})}  \frac{O(n)}{\theta(n, j_i, n)}\right) \langle \Gamma_{j_1, j_2, \ldots, j_t} \rangle
\end{equation} 
 To further evaluate the Kauffman bracket of the colored trivalent graph $\Gamma$ in equation \eqref{fe}, we work with the theory of Knotted Trivalent Graphs (KTG's for short) \cite{Thurston}.  

\begin{defn} \label{d.ktg} A Knotted Trivalent Graph (KTG) is a trivalent framed graph $\Gamma$ along with a coloring $\sigma: E(\Gamma) \rightarrow \mathbb{N}$ from the set of edges of $\Gamma$ to the natural numbers, considered up to isotopy of the embedding into $\mathbb{R}^3$. 
\end{defn} 

We identify the colored trivalent graph $\Gamma$ considered in skein-theoretic evaluation of the colored Jones polynomials with a KTG by letting the arcs and the trivalent vertices form the 1-dimensional simplicial complex, and with the coloring by integers defining $\sigma: E(\Gamma) \rightarrow \mathbb{N}$. Framed links are special cases of KTGs with no vertices whose graphs are components of the link. We have the following set of KTG moves from the equivalences of Figure \ref{f.moves}: 
\begin{figure}[H]
\includegraphics[scale=.7]{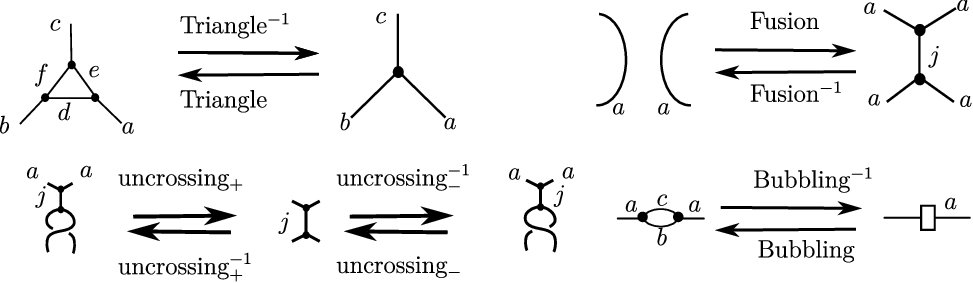} 
\caption{The KTG moves.}
\end{figure} 

\subsection{Converting a crossingless KTG to the theta graph.} \label{ss.convertktg}
Any KTG can be generated from the theta graph using the KTG moves \cite{Thurston}.   In particular, we can reduce a crossingless KTG $\Gamma$ such as the one in Equation \eqref{fe} from evaluating $\langle \mathbb{D}^n \rangle$ to the theta graph as follows.  \\
(1) Let $R$ be a bounded region in the trivalent graph $\Gamma$ with more than three edges on its boundary. Fuse pairs of edges on the boundary by applying the fusion equivalence, and repeat if necessary, to replace $\Gamma$ with a new trivalent graph and the bounded region $R$ with new region(s) $R_1, R_2, \ldots, R_k$, so that each $R_i$ has no more than three edges on its boundary. This is always possible since we can subdivide any polygon into triangles, though at the cost of introducing a sum over new fusion parameters as depicted in Figure \ref{f.moves}. Each term of the sum is now a new trivalent graph $\Gamma'$ whose bounded regions border at most three edges. See Figure \ref{f.4-edged} for an example. 
\begin{figure}[H]
\includegraphics[scale=.7]{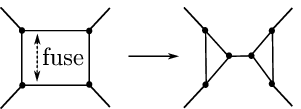} 
\caption{\label{f.4-edged} Fusing an opposite pair of edges reduces a bounded region of 4 sides to two with 3 sides.}
\end{figure} 
\noindent (2) Contract every bounded region (which is now either a triangle or a bigon) in $\Gamma'$ to a vertex by applying the inverses of the Triangle move and the Bubbling move, multiplying by a tetrahedron coefficient over the theta function each time.

\subsection{The limiting skein element $L^n_J$ as a KTG}
In this section we discuss how to obtain $L^n_J=L^n_J(D)$ for any link $L$ with diagram $D$ from the theta KTG by applying KTG moves. We represent $L^n_J$ in terms of a planar weighted graph $G$. 

A \textit{weighted graph} is a graph with a weight on every edge. Let $G$ be a planar $\mathbb{Z}\setminus \{0\}$-weighted  graph. Consider the surface $F_G$ obtained by replacing each vertex of $G$ by a disk and each edge $E$ with weight $w$ by a twisted band of $|w|$ half twists (right-handed if $w > 0$, and left-handed if $w < 0$). Then $\partial(F_G)$ defines a diagram for a link $L$. Conversely, to every link $L$ we can associate a planar, 2-connected, $\mathbb{Z}\setminus \{0\}$-weighted graph $G(L)$, so that $\partial(F_{G(L)})$ is a diagram of $L$. 

For every link $L$ with diagram $D$, recall from Definition \ref{d.cjp} that we can find the $n+1$ colored Jones polynomial of $L$ by evaluating $\langle \mathbb{D}^n \rangle$. Moreover, the evaluation is simplified to that of finding the Kauffman bracket of the theta graph by fusing two $n$ colored strands of an $n$-cabled twist region in $\mathbb{D}^n$, then untwisting the crossings to obtain a new sum (Eq. \eqref{fe}) over trivalent graphs indexed by fusion parameters $j_1, \ldots, j_t$.

 The fusion move introduces a sum over fusion parameters $0\leq j \text{ even } \leq 2n$. Another way of describing the limiting skein element $L^n_J$ (Definition \ref{d.limitingskein}), is that it corresponds to the term in the sum of Equation \eqref{fe} with the fusion parameter equal to $2n$ for each twist region.

 \begin{figure}[H]
 \includegraphics[scale=.7]{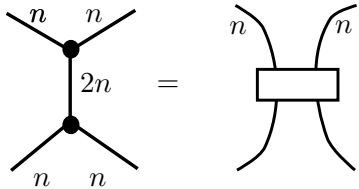}
\caption{Fusion parameter $j=2n$.}
 \end{figure}

\begin{lem} \label{l.genKTG}
Fix an integer $n\geq 1$. Let $S= L^n_J = L^n_J(D)$ be the skein element defined from a diagram $D= \partial(F_{G(L)})$ of $L$ as in Definition \ref{d.limitingskein}, then $S$ can be generated from the theta KTG using a sequence of triangle moves, bubbling moves, and  inverses of the fusion move.  
\end{lem} 

\begin{proof}

Given $S = L^n_J$ consider the graph $G$ corresponding to $L$. We can obtain $S$ from $G$ by replacing each vertex $v$ of $G$ with valence $|v|$ by a skein element consisting of $|v|$ $2n$ Jones-Wenzl projectors arranged  cyclically $\jw_{2n}^1, \jw_{2n}^2, \ldots, \jw_{2n}^{|v|}$, with $\jw_{2n}^{i}$ joined with $\jw_{2n}^{i-1}$ and $\jw_{2n}^{i+1}$ for $1 < i < |v|$, and $\jw_{2n}^{1}$ joined with $\jw_{2n}^{|v|}$ and $\jw_{2n}^{2}$ as shown in an example in Figure \ref{f.G}.  The valence of a vertex in $G$ corresponds to the number of Jones-Wenzl projectors abutting a bounded region in the complement of $S$ in the plane. 

\begin{figure}[H]
\includegraphics[scale=.2]{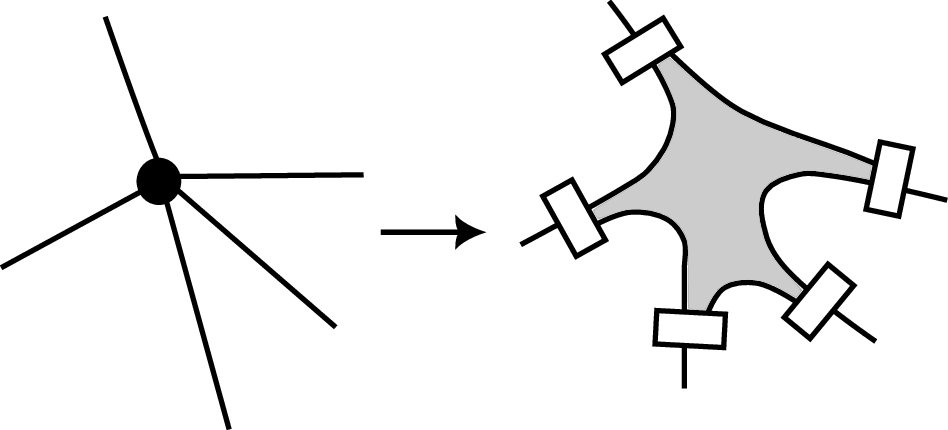} 
\caption{\label{f.G} Left: A vertex of valence 5 in $G$. Right: Corresponding skein element picture with 5 Jones-Wenzl projectors abutting the shaded region.}
\end{figure} 

We first convert $S$ to a KTG $S'$ using only the fusion move. 

\noindent \textbf{Case 1:} If every vertex of $G$ has valence 3, then $L^n_J$ is already a KTG without crossings, and we can reduce it to the theta KTG as described in Section \ref{ss.convertktg}. 

\noindent \textbf{Case 2:} Suppose a region corresponding to a vertex in $G$ has more than three Jones-Wenzl projectors abutting the region, corresponding to a vertex in $G$ with valence $|v| >3$. We fuse two edges adjacent to an edge connecting a pair of projectors. This incorporates the pair of projectors into a new trivalent vertex, with the remaining $|v|-1$ projectors abutting a new region $R'$.  We repeat the procedure with the remaining projectors that are not part of  a trivalent vertex, gradually converting the entire graph into a KTG.  Call the new KTG $S'$, we reduce it to the theta KTG again by the steps described in Section \ref{ss.convertktg}. 

Reverse all moves above to get that the theta graph generates $S$ via a sequence of only triangle moves and the inverses of the fusion move. Note the untwisting move was not needed because $S$ and therefore $S'$ has no crossings. 

\end{proof} 

\begin{defn}{(\cite[Definition 6]{vdV2009})} Define the unnormalized $N$ colored Jones invariant $\langle \Gamma \rangle_N(A)$ of a KTG $\Gamma$ to be the Kauffman bracket of the skein element obtained from a diagram of $\Gamma$ in the plane by replacing every edge by $n$ parallel edges joined by a $n$ Jones-Wenzl projector and every vertex by a trivalent skein vertex (as on the left in Figure \ref{f.triv}). 
\end{defn}

Although not used in this paper, van der Veen showed $\langle \Gamma \rangle_N(A)$ is a well-defined invariant of KTG's under trivalent isotopy moves \cite[Proposition 1]{vdV2009}. 

\begin{defn}
Define the normalized (reduced) colored Jones invariant of a KTG $\Gamma$ with $s$ split components to be $\widehat{J}_{N=n+1}(\Gamma) = \langle \Gamma \rangle_N/ \langle U^s \rangle_N$, where $U^s$ is the $s$-component unlink.  
\end{defn}

Similar to the colored Jones invariant for KTGs, we define $\widehat{J}^{\infty}_{n+1}(L; A) = \langle L^n_J \rangle / \langle \jwc_n \rangle$.

\subsection{Proof of Theorem \ref{t.volume}}

In this section we prove Theorem \ref{t.volume}. The argument reduces to evaluating the Kauffman bracket of  trivalent graphs in which every edge is colored by $2n$, and no fusion terms are present since we assume that the link is obtained from the $2n$ colored theta graph only by applying the triangle move. See Figure \ref{f.linkr} below for an illustration on the link in Figure \ref{f.link}.
\begin{figure}[H]
\includegraphics[scale=.7]{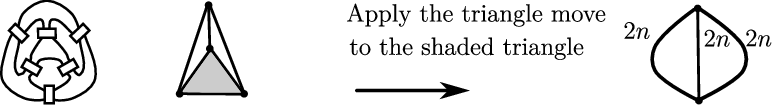}
\caption{\label{f.linkr}Reducing $L^n_J$ to a theta graph using only the triangle move. In the middle figure every edge is labeled with $2n$.}
\end{figure}  

Comparing with van der Veen's work in \cite{vdV2009}, for an augmented KTG he selects a particular term in the sum from fusion with parameter $j_1=n, j_2=n, \ldots, j_t=n$ when $n$ is even by augmentation. In our setting we do not need to do this because the limiting skein element $L_J^n$ does not have any crossings to begin with. The rest of the proof follows a similar argument to van der Veen, except that we use recent results \cite{BDK}, \cite{CM} to evaluate the polynomial at a root of unity $e^{ \frac{\pi i}{2N-1}}$ near the root of unity $e^{ \frac{\pi i}{2N}}$ considered by Conjecture \ref{c.volume}.

\tvolume*

\begin{proof} 
We compute $\widehat{J}_{N=n+1}( L^n_{J}) $ as a sum of rational functions involving the tetrahedron coefficient and theta functions.  By assumption, from the $2n$ colored theta graph we arrive at $L^n_J$ in a sequence of moves  containing only triangle moves. Therefore, $L^n_J$ is a KTG and $\widehat{J}^{\infty}_{N}( L^n_{J})$ is the normalized $n$ colored Jones invariant of $L^n_J$.  Since $L^n_J$ contains a single split component, we have 
\[ \widehat{J}_N^\infty(L; A)  = \begin{pmatrix} \frac{\tetran }{\theta(2n, 2n, 2n)} \end{pmatrix}^T \frac{\theta(2n, 2n, 2n)}{\langle \jwc_{n} \rangle}.  \] 

Again we note that the entries of our tetrahedron coefficients and the $6j$-symbols are by color rather than by the dimension of the irreducible representation as in \cite{vdV2009}. Through a similar computation to \cite{vdV2009}, with the dimension $N' = 2n+1$ corresponding to the color $2n$, we get 
\[ \lim_{N\rightarrow \infty}  (N')^{-1} \log \bigg \lvert \frac{\theta(2n, 2n, 2n)}{\langle \jwc_n \rangle} \bigg\rvert_{e^{\frac{\pi i}{N'}}} \bigg \rvert = \lim_{N\rightarrow \infty} (N')^{-1} \log \bigg \lvert\frac{O(2n)}{O(n)} \bigg\rvert_{e^{\frac{\pi i}{N'}}}  \bigg\rvert = 0. \]

This is because $e^{\frac{\pi i}{N'}}$ is a $2N'$th-root of unity, which implies $O(N+k-1) = (-1)^n O(k-1) = O(N-k-1)$, and these relations can be used to simplify $\frac{\theta(2n, 2n, 2n)}{O(n)} = \frac{O(2n)}{O(n)}$ and  $\frac{\theta(2n, 2n, 2n)}{O(2n)} = \pm 1$ at $A =  e^{\frac{\pi i}{N'}}$. 
Therefore 
\begin{align*}  & \lim_{N\rightarrow \infty} \frac{ \log \bigg \lvert \sixj \bigg \rvert_{e^{\frac{\pi i}{N'}}} \bigg \rvert}{N'} = \\
& = \lim_{N\rightarrow \infty} (N')^{-1} \log \left | \left. \frac{O(2n) \tetran}{(\theta(2n, 2n, 2n))^2} \right | _{e^{\frac{\pi i}{N'}}}  \right | =  \lim_{N\rightarrow \infty} (N')^{-1} \log  
\left | \left. \frac{ \tetran}{\theta(2n, 2n, 2n)} \right | _{e^{\frac{\pi i}{N'}}} \right |. 
\end{align*} 

Costantino \cite{Costantino} showed $\displaystyle \lim_{N'\rightarrow \infty} \frac{\pi}{2N'} \log 
\left \lvert  \left. \sixj \right \rvert_{e^{\frac{\pi i}{2N'}}} \right \rvert = v_8$.  The root of unity $e^{\frac{\pi i}{2N'}}$ can be replaced by its square $e^{\frac{\pi i}{N'}}$, see \cite[Lemma 3.13]{BDK} and \cite[Theorem 1]{CM}, to give
\[ \lim_{N'\rightarrow \infty} \frac{\pi}{N'} \log\left \lvert \left. \sixj \right\rvert_{e^{\frac{\pi i}{N'}}}  \right \rvert = v_8. \]

Thus with $N' = 2n+1 = 2(n+1)-1 = 2N-1$ we have 
\[  2\pi \lim_{N\rightarrow \infty} \frac{\log \left \lvert  \widehat{J}_N^\infty(L; e^{\frac{\pi i}{N'}}) \right \rvert }{N'} 
 =  T \lim_{N\rightarrow \infty} \frac{2\pi}{2N-1} \left. \log \begin{vmatrix}  
\sixj \end{vmatrix}_{e^{\frac{\pi i}{2N-1}}} \right \rvert = 2T v_8. \]

\end{proof}

\bibliographystyle{plain}
\bibliography{references}

\end{document}